\documentclass[a4paper,12pt]{article}

\usepackage{epsfig}
\usepackage{amsfonts,eucal}
\usepackage{amssymb,amsmath,amsthm,color}
\usepackage[normalem]{ulem}

\usepackage[T2A]{fontenc}
\usepackage[cp1251]{inputenc}
\newtheorem{remark}{Remark}[section]
\newtheorem{theorem}{Theorem}[section]
\newtheorem{proposition}{Proposition}[section]
\newtheorem{lemma}{Lemma}[section]

\newcommand{\Ga}{\Gamma}
\newcommand{\ga}{\gamma}
\newcommand{\BB}{{\mathbb B}}
\newcommand{\XX}{{\mathbb X}}
\newcommand{\R}{{\mathbb R}}
\newcommand{\N}{{\mathbb N}}
\newcommand{\Z}{{\mathbb Z}}
\newcommand{\X}{{\R^d}}
\newcommand{\XXX}{{\mathfrak{X}}}
\newcommand{\La}{\Lambda}
\newcommand{\la}{\lambda}
\newcommand{\B}{\mathcal{B}}

\newcommand{\M}{\mathcal{M}}
\newcommand{\CB}{\mathcal B}

\begin{document}

\title{Invariant measures for contact processes with state dependent birth and death rates. }

\author{Sergey Pirogov\thanks{
Institute for Information Transmission Problems, Moscow, Russia
(s.a.pirogov@bk.ru).} \and Elena Zhizhina\thanks{Institute for
Information Transmission Problems, Moscow, Russia (ejj@iitp.ru).} }

\date{}

\maketitle

\begin{abstract}
In this paper, we consider contact processes on locally compact separable metric spaces with birth and death rates heterogeneous in space. Conditions on the rates that ensure the existence of invariant measures of contact processes  are formulated.
One of the crucial condition is the so-called critical regime condition. To prove the existence of invariant measures we used our approach proposed in \cite{PZh}. We discuss in details the multi-species contact model with a compact space of marks (species) in which both birth and death rates depend on the marks.
\\

Keywords: multi-species continuous contact model, birth and death process in continuum, critical regime, correlation functions
\end{abstract}

\section{Introduction}

Contact processes have been widely used to describe evolutions and to predict a long-time behaviour in various models of population dynamics. Taking into account applications of the contact processes as models describing a spread of epidemic diseases or a population growth, one of the main problem under consideration is to determine the stationary regime and to prove the existence of invariant measures.
The contact processes on the lattice have been introduced in the pioneer papers of Harris \cite{H}, Holley and Liggett \cite{HL}, see also the monograph of Liggett \cite{Lig1985}. While in most of the works the contact processes were considered on the lattice, much of the interest in the recent years has focused on studying the contact processes running in continuous spaces, see e.g. \cite{KKP, KKPZ, KS}.
This class of processes is a particular case of continuous time birth and death processes.
One of the main features of the contact process is the clustering of the system, i.e. particles are grouped into clouds of high density located at large distances from each other.  It is worth noting that the appearance of a limiting invariant state is only possible in the so-called critical regime, that is, when there is a certain balance between birth and death.
As was shown in \cite{KKP}, there exists a continuum of invariant measures for the contact processes in $\X, \ d\ge 3,$ in the critical regime with a constant death rates. Clustering phenomenon is more visible in the case of small dimensions. It was proved in \cite{KKPZ} that for the contact processes in $\X, \ d= 1,2,$ invariant measures exist only if the dispersal kernel (the birth rate) has a heavy tail at infinity. In the case of light tails the pair correlation function grows to infinity as $t \to \infty$, and hence invariant measures do not exist. Thus heavy tails of dispersal kernels appear to make the critical regime more stable contrary to light tails.

The existence of invariant measures in the marked contact model in $\X, \ d\ge 3,$ with a compact spin space and for constant death rates was proved in \cite{KPZh}. The analogous multi-species model with immigration has been studied in \cite{PZh-MMJ}. Such models are used, in particular, to describe evolution in quasi-species populations with mutations, see \cite{Nov}.


In this work, we consider a class of contact processes running on locally compact separable metric spaces in the critical regime with state dependent birth and death rates.
The present paper is a generalization of our previous work \cite{PZh}, in which we formulated conditions providing the existence of invariant measures for contact processes on general spaces with constant death rates.
These invariant measures are described by a simple recurrent relation between their correlation functions and create a new class of point random fields.
Here we formulate conditions, including the critical regime condition, that ensure the existence of invariant measures for general contact processes with state dependent birth and death rates, and we prove the existence of a family of invariant measures.

Our approach is based on the analysis of the infinite system of hierarchical equations for correlation functions,  see e.g. \cite{KKP, PZh}. In Section 2, we introduce the model and formulate assumptions on the model, including the critical regime condition.
We formulate the main result in Section 3. Section 4 contains the proof of the main theorem. In Section 5 we apply general results of Sect. 2 - 3 to the analysis of multi-component contact models in continuum with a compact space of marks and state dependent birth and death rates.

\section{The model}

In this section, we formulate assumptions on the model that provide the existence of invariant measures for the contact processes running in general spaces.

Let $\XXX$ be a locally compact separable metric space, ${\B}({\XXX})$ be its Borel $\sigma$-algebra,
and $m$ will denote a locally finite  Borel measure on ${\B} (\XXX)$, i.e. $m$ is finite on compact sets.
Denote by $\M(\XXX)$ the space of locally finite Borel measures on $\B(\XXX)$ and by ${\B}_{\mathrm{b}} ({\XXX})$
the system of all compact sets from ${\B}({\XXX})$.

A configuration $\gamma \in \Ga(\XXX)$ on $\XXX$ is a finite or countably infinite locally finite
unordered set of points in $\XXX$, and some of them can be multiple, i.e. repetitions are permitted.
If the measure $m \in \M(\XXX)$ is atomic then  $\Ga(\XXX)$ includes configurations with multiple points. Such situation will be on graphs were the measure $m$ is a counting measure.
As the phase space $\Ga$ of the continuous contact models, when $m$ is non-atomic (see e.g. \cite{KKP, KKPZ, KS}),
one can take the set of locally finite configurations in $\XXX$ with distinct elements:
\begin{equation} \label{confspace}
\Ga_c =\Ga_c\bigl(\XXX\bigr) :=\Bigl\{ \ga \subset \XXX \Bigm| \ | \ga\cap\La| < \infty, \ \mathrm{for \ all } \ \La \in {\B}_{\mathrm{b}}
(\XXX)\Bigr\},
\end{equation}
where $|\cdot|$ denotes the number of elements of a set.

We can identify each $\ga\in\Ga$ with an integer-valued measure $\sum_{x\in
\gamma }\delta_x\in \M(\XXX)$, where $\delta_x$ is
the Dirac measure with unit mass, and the sum is taken considering the multiplicity of elements in the configuration $\gamma$.
For any $ \La \in {\B}_{\mathrm{b}}(\XXX)$ we denote by $|\gamma \cap \La|$ the value $\gamma(\La)$ of the measure $\gamma$ on $\Lambda$.

The contact model is a continuous time Markov process on $\Gamma(\XXX)$ which is a
particular case of a general birth-and-death process. In this work we consider the process with non-homogeneous birth and death intensities.  The model is given by a heuristic
generator defined on a proper class of functions  $F:
\Gamma \to \R$ as follows:
\begin{align}
 \begin{aligned}\label{generator}
(L F)(\gamma) & = \sum_{ x \in \gamma} V(x) \left(F(\gamma \backslash x) - F(
\gamma)\right) \\
& +   \int\limits_{\XXX} \sum_{x \in \gamma}  a(y,x) (F(\gamma \,\cup
y ) - F(\gamma)) m(dy).
 \end{aligned}
 \end{align}
Notations $\gamma \backslash x$ and $\gamma \,\cup x$  in \eqref{generator} stand for removing and adding one particle at position $x \in \XXX$. Similarly, $x \in \gamma$ refers to any particle in the configuration $\gamma$.
The first term in \eqref{generator} corresponds to the death of a particle at position $x$: each element $x \in \gamma$ of the configuration $\gamma\in\Ga$ can die with the death rate $V(x)$. The second term of \eqref{generator} describes the birth of a new particle in a neighborhood $dy$ of the point $y$ with the
birth rate density $A(y,\gamma):=\sum_{x \in \gamma} a(y,x)$. In fact, we even do not know who is a parent, since the birth of a new particle at position $y$ has a cumulative rate $\sum_{x \in \gamma}  a(y,x)$.  Function $a(x,y)$ is called the dispersal kernel.
\medskip

One of the goals of this work is to formulate conditions on the rates in \eqref{generator} that guarantee the existence of the invariant measures of the corresponding contact processes. Crucial condition for the existence of the stationary regime in the contact process, as well as in other birth and death processes, is the so-called critical regime condition. This condition describes a "stochastic balance between birth and death".

In the case $V(x) \equiv 1$, i.e. when the death rate is a constant, the critical regime condition reads
\begin{equation}\label{cr1}
 \int\limits_{\XXX} a(x,y) \, \Psi(y) \, m(dy) \ = \  \Psi(x) \quad \mbox{for all } x \in \XXX,
\end{equation}
where $\Psi(x)$ is a strictly positive bounded measurable function.
Models under the critical regime condition in the form \eqref{cr1} were studied earlier, for $\Psi(x) \equiv 1$ in \cite{KKP} and for $\Psi(s)$ depending on the species variable $s$ in \cite{PZh}.
In \cite{KKP}, the contact model in continuum with $\XXX = {\mathbb{R}}^d$, $a(x,y) = \alpha(x-y)$ and $V(x) \equiv 1$ was considered, and for this model $\Psi(x) \equiv 1$. A more general case, when  $\XXX = {\mathbb{R}}^d \times S$ with a compact metric space $S$ (the space of species),  was studied in \cite{PZh}. Then we got that condition \eqref{cr1} holds for a function  $\Psi(x)= q(s)$  depending only on the species variable $s \in S$. The similar multi-species continuous contact model with a non constant function $V(s), \, s \in S,$  will be discussed later in Section \ref{MS}.
\medskip

In the case, when $V(x): \XXX \to  {\mathbb{R}}_+$ is a positive bounded function,
{\bf the critical regime} condition becomes the following: there exists a strictly positive bounded measurable function $\Psi(x), \; \Psi(x) \ge p_0>0$ such that
\begin{equation}\label{2a}
 \int\limits_{\XXX} a(x,y)\, \Psi(y) \, m(dy) \ = \ V(x) \, \Psi(x) \quad \mbox{for all } x \in \XXX.
\end{equation}
This condition is a detailed balance condition. It means that if the particles are initially distributed with the density $ \Psi(x)$, then this density is conserved under the dynamics, i.e. $ \Psi(x)$ is the density of the stationary distribution of particles.
Let us discuss the new form \eqref{2a} of the critical regime condition.
Condition \eqref{2a} can be rewritten in the following form:
\begin{equation}\label{cr}
 \int\limits_{\XXX} b(x,y)  \bar{\mathfrak{m}}(dy) \ = \ V(x) \quad \mbox{for all } x \in \XXX,
\end{equation}
with
\begin{equation}\label{nb}
\bar{\mathfrak{m}}(dy)  = \Psi(y) m(dy), \quad b(x,y) = \frac{a(x,y)}{\Psi(x)}.
\end{equation}
Consequently, assuming that the critical regime condition \eqref{2a} holds for the initial generator \eqref{generator} (i.e. for $a(x,y)$ and  $m(dy)$)  we can define a new measure
$$
\bar{\mathfrak{m}}(dy) = \Psi(y) m(dy)
$$
and a new intensity of birth
$$
b(x,y) = \frac{a(x,y)}{\Psi(x)},
$$
such that  \eqref{2a} will be rewritten as \eqref{cr}, and the generator \eqref{generator} will take the form
\begin{align}
 \begin{aligned}\label{generator-bis}
(L F)(\gamma) & = \sum_{ x \in \gamma} V(x) \left(F(\gamma \backslash x) - F(
\gamma)\right) \\
& +   \int\limits_{\XXX} \sum_{x \in \gamma}  b(y,x) (F(\gamma \,\cup
y ) - F(\gamma)) \bar{\mathfrak{m}}(dy).
 \end{aligned}
 \end{align}
It is worth noting that formulae \eqref{nb} imply that the second terms in the right-hand sides of equations \eqref{generator} and \eqref{generator-bis} are the same. The transition to the new measure $\bar{\mathfrak{m}}(dy)$ and to the new birth rate $b(x,y)$ is an analogue of the "ground state" transformation in quantum mechanics.

In our work, we always assume that the critical regime condition \eqref{2a} is satisfied.
Therefore, in what follows we will consider the generator $L$ given by the formula \eqref{generator-bis}, where $b$ and $V$ satisfy \eqref{cr}.
\medskip

Now we are ready to formulate conditions on the birth and death rates of the generator \eqref{generator-bis}.
We assume that the birth $b(y,x)$ and the death $V(x)$ rates of the contact process satisfy the following conditions:
\begin{enumerate}
\item{\it Measurability } condition:  $b:\XXX \times \XXX \to [0, \infty)$ is a non-negative bounded measurable function, and $V: \XXX \to (0, \infty)$ is a  strictly positive bounded measurable function
  \begin{equation}\label{V}
  0 \ < \  V_{{\rm min}} = \min\limits_{\XXX} V(x) \ \le \ V(x) \ \le  \ \max\limits_{\XXX} V(x) = V_{{\rm max}} \ < \ \infty;
\end{equation}
\item{\it Regularity} condition: there exists a constant $C>0$, such that
\begin{equation}\label{2a-bis}
\sup\limits_{x \in \XXX} \ \int\limits_{\XXX} b(y,x) \bar{\mathfrak{m}}( dy) \ < \ C;
\end{equation}
\item {\it Critical regime} condition:
\begin{equation}\label{cr-bis}
 \int\limits_{\XXX} b(x,y)  \bar{\mathfrak{m}}(dy) \ = \ V(x) \quad \mbox{for all } x \in \XXX,
\end{equation}
\item {\it Transience} condition.
Let us consider the continuous time jump Markov process with generator
\begin{equation}\label{L}
\mathcal{L} f(x) =  \int\limits_{\XXX} b(x,y) \big( f(y) - f(x) \big)  \bar{\mathfrak{m}}(dy).
\end{equation}
Then we assume that for two independent copies $X(t)$ and $Y(t)$ of this process starting with $X(0)=x$ and $Y(0)=y$ the following condition holds
\begin{equation}\label{2b}
\sup\limits_{x,y} \ \int_{0}^{\infty} \mathbb{E}_{x,y} b(X(t), Y(t)) dt < H
\end{equation}
with a constant $H>0$. Moreover, we assume that the integral in \eqref{2b} converges uniformly in $x$, $y$.
\end{enumerate}
\medskip
Note that
$$
\frac{b(x,y) \bar{\mathfrak{m}}(dy)}{V(x)} = \frac{a(x,y) \Psi(y) m(dy)}{\Psi(x) V(x)}
$$
defines the distribution of ancestors.

\begin{remark}\label{R2}
The sufficient condition for \eqref{2b} together with required uniform convergence reads
\begin{equation}\label{suf-cond}
\int\limits_0^{\infty}  \sup\limits_{x, y} \ \mathbb{E}_x b(X(t), y) dt < H
\end{equation}
\end{remark}

\begin{proof} Denote by  $p (x,dy,t)$ the transition function of the Markov jump process with generator \eqref{L} at time $t$. Then we get
$$
\sup\limits_{x,y} \ \int\limits_{0}^{\infty} \mathbb{E}_{x,y} b(X(t), Y(t)) dt =
\sup\limits_{x,y} \ \int\limits_0^{\infty} \int\limits_{\XXX} \int\limits_{\XXX} b(x', y') p(x, dx',t) p(y, dy',t) dt \le
$$
$$
\sup\limits_{y} \ \int\limits_0^{\infty} \int\limits_{\XXX}  \Big( \sup\limits_{x} \  \int\limits_{\XXX}  b(x', y') p(x, dx', t) \Big)  p(y, dy',t)  dt =
$$
$$
\sup\limits_{y} \ \int\limits_0^{\infty}  \int\limits_{\XXX}  \ \Big( \sup\limits_{x} \ \mathbb{E}_x b(X(t), y')  \Big)  p(y, dy',t) dt \le
$$
$$
\int\limits_0^{\infty} \sup\limits_{y} \ \int\limits_{\XXX}  \Big( \sup\limits_{y'} \sup\limits_{x} \ \mathbb{E}_x b(X(t), y') \Big)  p(y, dy',t) dt =
\int\limits_0^{\infty}  \sup\limits_{x, y'} \ \mathbb{E}_x b(X(t), y') dt.
$$
Therefore, condition \eqref{suf-cond} implies the uniform convergence in \eqref{2b}.
\end{proof}

\section{Time evolution of correlation functions. Main results}


Denote by ${\cal M}_{fm}(\Gamma)$ the set of all probability
measures $\mu$ which have finite local moments of all orders, i.e.
$$
\int_{\Gamma} |\gamma \cap \Lambda |^{n} \ \mu (d \gamma) \ < \ \infty
$$
for  all $\Lambda \in {\cal B}_b(\XXX)$ and $n \in N$.

Together with the configuration space $\Gamma$ we define the space  $\Gamma_0$ of finite configurations, and let $\Gamma^{(n)}_{0,\Lambda} = \{\eta \subset \Lambda: \ |\eta| = n  \}$ be the set of $n$-point configurations in $\Lambda \in {\cal B}_b(\XXX)$.
If a measure $ \mu \in {\cal M}_{fm}(\Gamma)$ is locally absolutely continuous with respect  to the Lebesque-Poisson measure
$$
\lambda_{z, \, \bar{\mathfrak{m}}} \ = \ \sum\limits_{n=0}^{\infty} \frac{z^n}{n!} \, {\bar{\mathfrak{m}}}^{\otimes n}, \quad \mbox{i.e. } \;
\lambda_{z, \, \bar{\mathfrak{m}}}(\Gamma^{(n)}_{0,\Lambda}) \ = \ \frac{z^n \,(\bar{\mathfrak{m}}(\Lambda))^n}{n!},
 \; n=0,1, \ldots,
$$
where $\bar{\mathfrak{m}}(\Lambda)= \int_{\Lambda} \bar{\mathfrak{m}}(dx)$, then there exists the corresponding system of the correlation functions, i.e.
densities of the correlation measure with respect to the Lebesque-Poisson measure.
The terminology originates in statistical mechanics, see, for instance, \cite[Ch. 4]{R}.
Denote by ${\cal M}_{\rm{corr}}(\Gamma)$  the subclass of the class ${\cal M}_{fm}(\Gamma)$ consisting of probability measures on $\Ga$ for which correlation functions exists.

The evolution equation for the system of $n$-point correlation functions corresponding to the continuous contact model in $\XXX$ has the following recurrent forms, see e.g. \cite{KKP, PZh}:
\begin{equation}\label{59}
\frac{\partial k_{t}^{(n)}}{\partial t} \ = \ \hat L_n^{\ast} k_{t}^{(n)} \
+ \ f_{t}^{(n)}, \quad n\ge 1; \qquad k_{t}^{(0)} \equiv 1,
\end{equation}
where
\begin{align}
 \begin{aligned}\label{korf}
 \hat L^{\ast}_n k_t^{(n)}(x_1, &\ldots, x_n)  =  - \Big( \sum_{i=1}^n  V(x_i) \Big) k_t^{(n)}(x_1,
\ldots, x_n) \\
& + \sum_{i=1}^n \int\limits_{\XXX} b(x_i, y) k_t^{(n)}(x_1, \ldots, x_{i-1}, y,
x_{i+1}, \ldots, x_n) \bar{\mathfrak{m}}(dy).
\end{aligned}
\end{align}
Here $f_{t}^{(n)}$ are functions on $\XXX^{n}$ defined for $n \ge 2$ by
\begin{equation}\label{f}
f_{t}^{(n)}(x_1, \ldots, x_n) \ = \ \sum_{i=1}^n  k_{t}^{(n-1)}(x_1,
\ldots,\check{x_i}, \ldots, x_n) \sum_{j\neq i} b(x_i, x_j),
\end{equation}
and  $f_{t}^{(1)} \equiv 0$. The notation $\check{x_i}$ means that this coordinate is excluded.

Let $\XX_{n} = \BB({\XXX}^n)$ be the Banach space of all measurable real-valued bounded functions on $\XXX^n$ with the $\sup$-norm.
Consider the operator $\hat L_n^{\ast}$ as an operator on the Banach space $\XX_{n}$ for any $n\geq 1$. Then it is a bounded linear operator in $\XX_{n}$, and the arguments based on the variation of parameters formula yields that
\begin{equation}\label{61A}
k_{t}^{(n)} \ = \ e^{t \hat L_n^{\ast}} k_{0}^{(n)} \ + \  \int\limits_0^t e^{(t-s) \hat
L_n^{\ast}} f_s^{(n)} \ ds,
\end{equation}
where $f_s^{(n)}$ is expressed through $k_s^{(n-1)}$ by (\ref{f}).
Thus, the solution to the Cauchy problem \eqref{59} in  $\XX_{n}$ with arbitrary initial values $k_{0}^{(n)}\in\XX_{n}$ exists and is unique provided $f_{t}^{(n)}$ is constructed recurrently via the
solution to the same Cauchy problem \eqref{59} for $n-1$.

The goal of this paper is to prove the existence of a family of invariant measures for the contact process in the critical regime generated by the operators of the form \eqref{generator-bis}.
These measures are described in terms of the corresponding correlation functions
$\{k^{(n)}\}_{n\geq 0}$ as solutions to the following system: 
\begin{equation}\label{Last}
\hat L^{\ast}_n k^{(n)} + f^{(n)}=0, \quad n \ge 1, \quad
k^{(0)}\equiv 1,
\end{equation}
where $\hat L_n^{\ast}, \, f^{(n)}$ were defined by \eqref{korf}-\eqref{f}.
In the sequel, we say that $k:\Ga_{0}\to \R$ solves the system \eqref{Last} in the Banach spaces $(\XX_{n})_{n\geq 1}$ if the corresponding $k^{(n)}\in\XX_{n}$, $n\geq 1$ solve \eqref{Last}.

The main result of the paper is the following theorem.

\begin{theorem}\label{main} {\it Assume that the birth rates $b(y,x)$ and the death rates $V(x)$ of the contact process
satisfy measurability, regularity \eqref{2a-bis}, critical regime \eqref{cr-bis} and transience \eqref{2b} conditions.
Then the following assertions hold.

{(i)} For any positive constant $\varrho >0$ there exists a probability measure $\mu^{\varrho} \in {\cal M}_{\rm{corr}}(\Gamma)$ on $\Ga$ such that its
correlation function $k_{\varrho}: \Ga_{0}\to\R_{+}$ solves (\ref{Last})  in the Banach spaces $(\XX_{n})_{n\geq 1}$, and  the corresponding system $\{k_{\varrho}^{(n)}\}_{n\geq 1}$ satisfies $k_\varrho^{(1)}\equiv \varrho$. Moreover, the following bounds hold for all $(x_1, \ldots, x_n)\in {\XXX}^{n}$
\begin{equation}\label{estimate}
k^{(n)}_\varrho (x_1, \ldots, x_n) \ \le   D  {H}^n (n!)^2 \qquad \text{with } \; D = \sum\limits_{n=1}^{\infty} \frac{(\varrho/ H)^n}{(n!)^2}
\end{equation}
where $H$ is the same constant as in  (\ref{2b}).

{(ii)} Let $\{k_{\varrho, t}^{(n)}\}_{n\geq 1}$ be the
solution to the Cauchy problem (\ref{59}) with
 initial data $k_0 = \{k_0^{(n)} \}$ corresponding to the Poisson measure $\pi_\varrho$ with intensity $\varrho$:
\begin{equation}\label{k0}
k_0^{(0)}= 1, \quad k_0^{(n)}(x_1, \ldots, x_n) = \varrho^n,  \; n\ge 1.
\end{equation}
Then
\begin{equation}\label{Th1-2}
\| k_{\varrho, t}^{(n)} \ - \ k_\varrho^{(n)} \|_{\XX_n} \ \to \ 0, \quad t \to \infty, \quad \forall n\geq 1.
\end{equation}
}
\end{theorem}

\medskip
The main strategy of the proof follows the same line as in \cite{PZh}. However, in the present paper we should modify some steps of the previous proof for contact processes with spatially non-homogeneous rates.

\section{The proof of Theorem \ref{main}}


For the first correlation function $k^{(1)}$ we get from \eqref{korf}  and (\ref{Last}) the following equation
\begin{equation}\label{8}
-V(x) k^{(1)}(x) + \int\limits_{\XXX} b(x,y) k^{(1)}(y) \bar{\mathfrak{m}}(dy) = 0,
\end{equation}
that can be written using the critical regime condition \eqref{cr-bis} as
\begin{equation}\label{8-bis}
\int\limits_{\XXX} b(x,y) \big( k^{(1)}(y)- k^{(1)}(x) \big) \bar{\mathfrak{m}}(dy) = 0.
\end{equation}
Clearly $k^{(1)}(x) \equiv \varrho $ is an element of $\XX_{1}$ and it solves \eqref{8-bis} (and \eqref{8}).
We notice that $\varrho$ can be interpreted as the spatial density of particles.

In the proof of the first statement of Theorem \ref{main} we use the
induction in $n\in\N$. If for any $n>1$ we succeed to
solve equation (\ref{Last}) and express $k^{(n)}$ through
$f^{(n)}$, then knowing the expression of $f^{(n)}$ through
$k^{(n-1)}$ (see (\ref{f})), we get the solution $\{k^{(n)}\}_{n\geq 1}$ to the full system
(\ref{Last}) recurrently.
%

\begin{lemma} \label{3.2}
The operator $e^{t \hat L_n^{\ast}}$, where $\hat L_n^{\ast}$ was defined in \eqref{korf}, is
positive, i.e. it maps non-negative functions to non-negative functions.
\end{lemma}
\begin{proof}
The operator
$$
A^i k^{(n)}(x_1, \ldots, x_n) \ := \ \int\limits_{\XXX} b(x_i, y) k^{(n)} (x_1, \ldots, x_{i-1}, y,
x_{i+1}, \ldots, x_n) \bar{\mathfrak{m}}(dy).
$$
is positive and bounded on $\XX_{n}$ for any $1\leq i\leq n$.
Set
\begin{align}
\begin{aligned}\label{21}
{\mathcal{L}}^{i} k^{(n)}(x_1, \ldots, x_n) \ = &\int\limits_{\XXX} b(x_i, y) k^{(n)}(x_1, \ldots, x_{i-1}, y, x_{i+1}, \ldots,
x_n) \bar{\mathfrak{m}}( dy)\\
& - V(x_i) k^{(n)}(x_1, \ldots, x_n).
\end{aligned}
\end{align}
Using the Trotter formula for the sum $A+B$ of two bounded operators:
$$
e^{t(A+B)} = \lim\limits_{n \to \infty} \Big( e^{\frac{t A}{n}} e^{\frac{t B}{n}} \Big)^n
$$
we conclude that
\begin{equation}\label{20L}
e^{t \, {\mathcal{L}}^{i} }f = \lim\limits_{n \to \infty} \Big( e^{t \frac{ A^{i}}{n}} e^{-t \frac{V}{n}} \Big)^n f \ge  e^{- t \, V_{\rm{max}}} \, e^{t \, A^{i}} f \ge 0
\end{equation}
for any non-negative function $f$. Here $V$ is the operator of multiplication on the positive bounded function $V$,
and constant $ V_{\rm{max}}$ was defined in \eqref{V}.

Representation \eqref{korf} yields
$$
e^{t \hat L_n^{\ast}} \ = \ \otimes_{i=1}^n \  e^{t \, {\mathcal{L}}^{i}}.
$$
Then taking into account that
\begin{equation}\label{20L}
\otimes_{i=1}^n \ e^{- t \, V_{\rm{max}}} \, e^{t \, A^{i}}
\end{equation}
is a positive operator, we get the desired conclusion.
\end{proof}

Let $k^{(1)}$ be any positive constant.
Next we will construct a solution to the system (\ref{Last}) satisfying estimates \eqref{estimate}.
As follows from (\ref{f}), the function $f^{(n)}$ is the sum of
functions of the form
\begin{equation}\label{32}
f_{i,j} (x_1, \ldots, x_n)  =  k^{(n-1)} (x_1,\ldots,\check{x_i},
\ldots, x_n)  b(x_i, x_j), \quad i\neq j.
\end{equation}

We suppose by induction that
$$
k^{(n-1)} (x_1, \ldots, x_{n-1}) \ \le \ K_{n-1}, \quad \text{for all } \; (x_1, \ldots, x_{n-1})\in {\XXX}^{n-1},\quad n\geq 2,
$$
where $K_n = D C^n (n!)^2$, and $D, C$ are some constants. Consequently,
\begin{equation}\label{34}
f_{i,j}(x_1, \ldots, x_n) \ \le \  K_{n-1} b(x_i, x_j),\quad (x_1, \ldots, x_{n})\in {\XXX}^{n}.
\end{equation}
Using the positivity of the operator $e^{t \hat L_n^{\ast}}$ and (\ref{34}) we have
\begin{align}
\begin{aligned}\label{36}
\left(e^{t \hat L_n^{\ast}} f_{i,j} \right) (x_1, \ldots, x_{n})\ \le \  K_{n-1} \ \left(e^{t \hat L_n^{\ast}}
b(\cdot_i, \cdot_j) \right)(x_1, \ldots, x_{n}).
\end{aligned}
\end{align}
Using the critical regime condition \eqref{cr-bis} we conclude that $e^{t {\mathcal{L}}^{i}}1\!\!1=1\!\!1$, $\forall i=1,\,\ldots, n,$ where ${\mathcal{L}}$ was introduced by \eqref{21} and
 $1\!\!1(x)\equiv 1$. Thus we get
\begin{equation}\label{36_1}
\left(e^{t \hat L_n^{\ast}}
b(\cdot_i, \cdot_j)\right) (x_1, \ldots, x_{n})\  =  \
\left(e^{t ( {\mathcal{L}}^i + {\mathcal{L}}^j)} b(\cdot_i, \cdot_j)\right) (x_1, \ldots, x_{n}).
\end{equation}
Note that the latter function depends only on variables $x_{i}$ and $x_{j}$.
\medskip

Notice that $e^{t \hat L_n^{\ast}} f_{i,j}$ is integrable with respect to $t$ on $\R_{+}$. Indeed, relations  \eqref{20L}, \eqref{32}, condition \eqref{2b} and the identity
\begin{equation}\label{2BB}
 e^{t \hat L_n^{\ast}} b(x,y)   =   \mathbb{E}_{x,y} b(X(t), Y(t))
\end{equation}
imply that
\begin{equation}\label{vij}
v^{(n)}_{i,j} \ = \ \int_0^{\infty} e^{t \hat L_n^{\ast}} f_{i,j} \ dt \
\le  K_{n-1} \, H,
\end{equation}
where $H$ is the same constant as in \eqref{2b}.
\medskip

Starting from now, the proof of the main result completely repeats the reasoning given in the proof of Theorem 3.1 from \cite{PZh}. We present next steps of the proof here for the reader's convenience.
We denote
\begin{equation}\label{k-def}
v^{(n)}  = \sum_{i \neq j} v^{(n)}_{i,j}=\int_0^{\infty} e^{t \hat L_n^{\ast}} f^{(n)} dt, \qquad f^{(n)} = \sum_{i \neq j} f_{i,j},
\end{equation}
where $f_{i,j}$ was defined by \eqref{32}. Next we prove that function $v^{(n)}$ is  a solution to \eqref{Last} in $\XX_{n}$.
It is easily seen from \eqref{vij} and induction procedure that $v^{(n)}\in\XX_{n}$. Since $e^{t \hat L_n^{\ast}}$ is a strongly continuous semigroup we have
\begin{equation}\label{37A}
e^{t \hat L_n^{\ast}}f^{(n)}-f^{(n)}=\hat L_n^{\ast}\int_{0}^{t}e^{s \hat L_n^{\ast}}f^{(n)}ds.
\end{equation}
Rewrite \eqref{37A} as
\begin{equation}\label{37A-bis}
e^{t \hat L_n^{\ast}}f^{(n)} = f^{(n)}+ \hat L_n^{\ast}\int_{0}^{t}e^{s \hat L_n^{\ast}}f^{(n)}ds.
\end{equation}
Then using condition \eqref{2b}, inequality \eqref{34}, Lemma \ref{3.2} and the fact that $\hat L_n^{\ast}$ is a bounded operator we conclude that the right hand side of \eqref{37A-bis} has a uniform in $x_1, \ldots, x_n$ limit as $t \to \infty$, therefore, the left hand side of \eqref{37A-bis}, i.e. $ e^{t \hat L_n^{\ast}}f^{(n)}$, also converges in $\XX_{n}$. Moreover the limit is a nonnegative function in  $\XX_{n}$. However, if this function is somewhere strictly positive, then we get a contradiction with condition \eqref{2b}, since in this case the integration over $t$ will be unbounded. Thus, we conclude that the following limit holds in $\XX_{n}$:
\begin{equation}\label{2c-bis}
 e^{t \hat L_n^{\ast}}f^{(n)} \to 0, \quad t\to \infty.
\end{equation}
A passage to the limit in \eqref{37A}  as $t\to\infty$ together with \eqref{2c-bis} shows  that $v^{(n)}$ defined in \eqref{k-def} can be taken as a solution $k^{(n)}$ to \eqref{Last} in $\XX_{n}$.

Since the function $f^{(n)}$ is the sum
of functions $f_{i,j}$, $i\neq j$ we deduce from \eqref{vij} that $ v^{(n)}$
is bounded by $n^2  K_{n-1} H$. Thus we get the recurrence inequality
\begin{equation}\label{50}
K_n \ \le \ n^2 K_{n-1} H,
\end{equation}
and by induction it follows that
\begin{equation}\label{49}
K_n \ \le \ H^n \, (n!)^2 \, k^{(1)}.
\end{equation}
Thus this solution $k^{(n)} = v^{(n)}$ satisfies estimate
\begin{equation}\label{49A}
v^{(n)} (x_1, \ldots, x_n) \ \le \ H^n \, (n!)^2 \,  k^{(1)}.
\end{equation}

Of course, any family of function of the form
$$
k^{(1)}\equiv \varrho,\quad k^{(n)}   =  \int\limits_0^{\infty} e^{t \hat
L_n^{\ast}} f^{(n)} \ dt  + A_n,\quad n \geq 2,
$$
with an arbitrary constant  $A_n$ is a solution to the system
(\ref{Last}) too.  Here $ f^{(n)}$ is defined as above with the help of \eqref{32}.  Taking $A_n=\varrho^n$ we conclude that
\begin{equation}\label{52}
k^{(1)}_\varrho\equiv\varrho,\quad k^{(n)}_\varrho = \int\limits_0^{\infty} e^{t \hat L_n^{\ast}} f^{(n)}dt  +  \varrho^n,\quad n\geq 2,
\end{equation}
is the desired solution to the stationary problem \eqref{Last} in the Banach spaces $(\XX_{n})_{n\geq 1}$.
To emphasize the dependence of $f^{(n)}$ on $\varrho$, we will use notation $f_{\varrho}^{(n)}$  for $f^{(n)}$.
For the solutions $\{ k_{\varrho}^{(n)} \}_{n\geq 1}$ of \eqref{52} instead of (\ref{50}) we have the recurrence
\begin{equation}\label{53}
K_n \ \le \  n^2 K_{n-1} H \ + \ \varrho^n.
\end{equation}
Taking $L_n = \frac{K_n}{H^n (n!)^2}$ we get from \eqref{53}
$$
L_n \le L_{n-1}+ \frac{\varrho^n}{H^n (n!)^2} \le D \quad \forall \; n=1,2, \ldots; \qquad L_0=0.
$$
This yields
\begin{equation}\label{55}
K_n \ \le \ D H^n (n!)^2 \qquad \mbox{with} \; D = \sum\limits_{n=1}^{\infty} \frac{(\varrho / H)^n}{(n!)^2}.
\end{equation}
\medskip

To be certain that the constructed system $\{ k_{\varrho}^{(n)} \}_{n\geq 1}$ is a system of correlation functions, i.e. it corresponds to a probability measure $\mu^\varrho$ on the configuration space $\Gamma$, we will
prove below that $\{ k_{\varrho}^{(n)} \}_{n\geq 1}$ can be constructed as the limit when $t \to \infty$ of the system
of correlation functions $\{ k_{t}^{(n)}\}_{n\geq 1}$ associated with the solution to the Cauchy problem \eqref{59} with
the initial data \eqref{k0}.


We recall that by the variation of parameters formula we have relation \eqref{61A} for the solution to the Cauchy problem (\ref{59}).
On the other hand, we proved above the existence of the solution $\{ k_{\varrho}^{(n)} \}_{n\geq 1}$ of the stationary problem:
\begin{equation}\label{New1}
 \hat L_n^{\ast}
k_\varrho^{(n)} \ = \ - f_\varrho^{(n)},
\end{equation}
with
$$
f_\varrho^{(n)}(x_1, \ldots, x_n) \ = \ \sum_{i,j:\ i\neq j}
k_\varrho^{(n-1)}(x_1, \ldots,\check{x_i}, \ldots, x_n) \ b(x_i, x_j).
$$
This solution is given by formula \eqref{52}, and
\eqref{New1} implies the following relation
\begin{equation}\label{New2}
\left( e^{t \hat L_n^{\ast}} -  1 \right) k_\varrho^{(n)} \ = \ -
\int\limits_0^t \frac{d}{ds} e^{(t-s) \hat L_n^{\ast}} k_\varrho^{(n)} ds \
\ = \ - \int\limits_0^t e^{(t-s) \hat L_n^{\ast}}  f_\varrho^{(n)} \ ds.
\end{equation}
Therefore from \eqref{61A} and \eqref{New2} we obtain
\begin{equation}\label{64}
k_{t}^{(n)} -  k_\varrho^{(n)} \ = \
e^{t \hat L_n^{\ast}}(k_{0}^{(n)} - k_\varrho^{(n)}) \ + \  \int\limits_0^t
e^{(t-s) \hat L_n^{\ast}} (f_{s}^{(n)} -  f_\varrho^{(n)}) \ ds.
\end{equation}
We will prove now that both terms in the right-hand side of (\ref{64}) converge to 0 in
the norm of $\XX_n$ as $t \to \infty$.

Formula (\ref{52}) yields
\begin{equation}\label{65}
e^{t \hat L_n^{\ast}} \big( k_{0}^{(n)} - k_\varrho^{(n)} \big) \ = \ - e^{t \hat
L_n^{\ast}} v^{(n)},
\end{equation}
where according \eqref{k-def} we have
\begin{equation*}\label{66}
v^{(n)} \ = \ \int_0^{\infty} e^{s \hat L_n^{\ast}} f_\varrho^{(n)}
\ ds.
\end{equation*}
Consequently, the first term in the r.h.s. of  \eqref{64} can be rewritten using \eqref{65} and \eqref{k-def} as follows
\begin{equation*}\label{1termA}
e^{t \hat L_n^{\ast}} \ v^{(n)} = \int_{0}^{\infty}e^{(t+s) \hat L_n^{\ast}}f_\varrho^{(n)} \, ds =  \int_{t}^{\infty}e^{r \hat L_n^{\ast}}f_\varrho^{(n)} \, dr.
\end{equation*}
Due to the structure \eqref{32} of the function $f_\varrho^{(n)}$ and the uniform convergence of the integral in \eqref{2b} we conclude that
\begin{equation}\label{1term}
||e^{t \hat L_n^{\ast}} \ v^{(n)}||_{{\XX}_n}\to 0,\quad t\to\infty.
\end{equation}
The second term in the r.h.s. of   \eqref{64} also tends to 0, and it can be proved in the same way as in our previous works, see e.g. \cite{KPZh}.

Thus we proved the strong convergence (\ref{Th1-2}), and the proof of the second part of Theorem \ref{main} is completed.
\medskip

Now we go back to the first part of the Theorem \ref{main}, and the final step of the proof is to show that the system of correlation functions $\{k_\varrho^{(n)} \}_{n \ge 1}$ corresponds to a probability measure $\mu^\varrho$ on the configuration space $\Gamma$. For this  we have constructed above $k_\varrho^{(n)}$ as the limit when $t \to \infty$ of solution $ k_{t}^{(n)} $
of the Cauchy problem (\ref{59}) with initial data (\ref{k0}):
\begin{equation}\label{limk}
k^{(n)}_\varrho \ = \ \lim_{t\to\infty} k_{t}^{(n)}.
\end{equation}
Then one can prove that solution $ k_{t}^{(n)}$
of the Cauchy problem satisfies the Lenard positivity and the moment growth conditions, see \cite{L1}-\cite{L2}.
The detailed proof of this fact can be found in \cite{KKPZ}.
Finally, these conditions imply that for any  $\varrho >0$  there exists a unique probability measure $\mu^\varrho \in {\cal M}_{\rm{corr}}(\Gamma)$, whose correlation functions are $\{ k^{(n)}_\varrho\}_{n \ge 1}$.
This completed the proof of Theorem \ref{main}.

\section{A multi-species (marked) continuous contact model in the critical regime.}\label{MS}

In this section we consider a continuous contact model where each element of the configuration is characterized by its location in the space as well as its mark. This model with constant mortality rates $V(x) \equiv 1$ was considered in \cite{KPZh, PZh-MMJ}, in the former paper the critical regime was studied, in the latter one, the subcritical regime.

The configuration  $\ga\in\Ga(\XXX)$ is a finite or countably infinity locally finite unordered set of points in $\XXX$, where $\XXX= \mathbb{R}^d \times S$ and $S$ is a compact metric space (the space of marks). The measure $m$ is taken in the form $m=l\otimes \nu$, where $l(dy) = dy$ is the Lebesgue measure on $ \mathbb{R}^d$ and $\nu(ds)$ is a finite Borel measure on $S$. We will use notations $x=(\xi, s), \; \xi \in  \mathbb{R}^d, \ s \in S$, for points $x \in \XXX$. We take the birth rates $a(x,x')$ in the following form
$$
a(x, x') \ = \ \alpha(\xi - \xi') Q(s,s'),
$$
where $\alpha(\cdot) \ge 0$ is a bounded measurable function satisfying normalization condition
\begin{equation}\label{norm}
\int_{\mathbb{R}^d} \alpha(u) du =1,
\end{equation}
and $Q: S \times S \to {\mathbb{R}}_+$ is a continuous (and so bounded) strictly positive function. For the death rate we assume that $V(\xi,s)= v(s)$, i.e. $V(x)$ depends only on the mark variable, and $v: S \to {\mathbb{R}}_+ $ is also a continuous strictly positive function. Consequently, the Krein-Rutman theorem implies that there are a positive number $r>0$ and a strictly positive continuous function $q(s)$ on $S$, such that
\begin{equation}\label{rq}
\int\limits_S \frac{Q(s, s')}{v(s)} q(s') \nu(d s') \ = \ r q(s)
\end{equation}
with $\; 0<q_{min} \le q(s) \le q_{max}<\infty$. If we suppose that $\Psi(x)$ depends only on $s \in S$,
then the critical regime condition \eqref{2a} is equivalent to the condition that
the compact positive operator $\tilde Q$ with the kernel $\frac{Q(s,s')}{v(s)}$ has the maximal positive eigenvalue $r=1$, i.e.
equality \eqref{rq} holds with $r=1$:
\begin{equation}\label{rq-1}
\int\limits_S \frac{Q(s, s')}{v(s)} q(s') \nu(d s') \ = \  q(s),
\end{equation}
and $\Psi(x) = q(s)$.

\begin{theorem}\label{T-multi}
Let $d \ge 3$, the birth rates have the form
$$
a(x, x') \ = \ \alpha(\xi - \xi') Q(s,s'), \quad \xi \in \mathbb{R}^d,  \;\; s \in S,
$$
where $\alpha(\cdot) \ge 0$ is a bounded measurable function satisfying condition \eqref{norm},
$Q: S \times S \to {\mathbb{R}}_+$ is a continuous strictly positive function. We assume that $V(\xi,s)= v(s)$ and $v: S \to {\mathbb{R}}_+ $ is a continuous strictly positive function:
$$
0< V_{min} \le v(s) \le V_{max} < \infty.
$$

Let the critical regime condition \eqref{rq-1} be true with $\int_S q(s) \nu (ds) =1$, and  $\Psi(x) = q(s)$.
Then all conditions of Theorem \ref{main} are fulfilled and, consequently, for any $\varrho >0$ there exists an invariant measure $\mu^\varrho$ whose correlation functions (w.r.t. the Lebesgue-Poisson measure with intensity $m(dy)$) satisfy the following estimates
\begin{equation}\label{T2-1}
k^{(n)}_\varrho (x_1, \ldots, x_n) \ \le \  D \, H^n \, (n!)^2 \, \prod\limits_{i=1}^n q(s_i)
\quad  \text{for all} \quad (x_1, \ldots, x_n)\in {\XXX}^{n},
\end{equation}
where $D$ was defined by \eqref{estimate}, and a constant $H$ is defined in the same way as in (\ref{2b}).

{(ii)} Moreover, the correlation functions $ k_\varrho^{(n)}, \, n =1,2,\ldots $, of the invariant measure $\mu^\varrho$ for any $\varrho>0$ can be constructed as the limit of correlation functions of the Cauchy problem with corresponding initial data. Indeed, for any $n=1,2, \ldots,$ the solution $k^{(n)}(t)$ of the Cauchy problem (\ref{59}) with initial data
\begin{equation*}\label{T2-k0}
k_0^{(0)}= 1, \quad k_0^{(n)}(x_1, \ldots, x_n) = \varrho^n \,\prod\limits_{i=1}^n q(s_i),
\end{equation*}
converges to the solution of the system \eqref{Last} of stationary (time-independent) equations as $t \to \infty$:
\begin{equation}\label{Th2-2}
\| k^{(n)}(t) \ - \ k_\varrho^{(n)} \|_{\XX_n} \ \to \ 0, \quad t \to \infty.
\end{equation}
\end{theorem}

\begin{proof}
For the proof, we consider the new measure
$
\bar{\mathfrak{m}}(dy) = \Psi(y) m(dy)
$
and the new  birth rates
$
b(x,y) = \frac{a(x,y)}{\Psi(x)},
$
given by formulas \eqref{nb} with  $\Psi(\xi, s) =  q (s)$, and then we  must check that all conditions of Theorem \ref{main} for $b$ and $\bar{\mathfrak{m}}(dy)$ are fulfilled.

The measurability and the regularity conditions are valid due to the boundedness of $\Psi$. Thus it remains to estimate ${\mathbb{E}}_x b(X(t), y)$ from above. Then bound \eqref{suf-cond} will imply the transience condition \eqref{2b}. Next Lemma guarantees the convergence of the integral in \eqref{suf-cond} in the case when $d \ge 3$.

\begin{lemma}\label{main-tr}
For all $t>0$ the following uniform upper bound holds
\begin{equation}\label{Eb-0}
{\mathbb{E}}_x b(X(t), y) \le \min \big\{ \| b \|_{\infty}, \ \frac{C}{t^{d/2}}  \big\}
\end{equation}
with a positive constant $C$.
\end{lemma}

\begin{proof}
Let us consider a Markov jump process on $\XXX$ with generator
\begin{equation*}\label{L}
\mathcal{L} f(x)\  =  \ \int\limits_{\XXX} b(x,y) \big( f(y) - f(x) \big)  \bar{\mathfrak{m}}(dy),
\end{equation*}
where $b(x,y)$ and $ \bar{\mathfrak{m}}(dy)$ were defined by \eqref{nb}.
The generator $L$ can be rewritten as
\begin{equation}\label{LV}
\mathcal{L} f(x) \ = \ V(x) \, \int\limits_{\XXX} \frac{b(x,y)}{V(x)} \big( f(y) - f(x) \big)  \bar{\mathfrak{m}}(dy).
\end{equation}
Using the critical regime condition \eqref{rq-1} with  $\Psi(\xi, s) =  q (s)$ we obtain
$$
 \int\limits_{\XXX} \frac{b(x,y)}{V(x)}\, \bar{\mathfrak{m}}(dy) = \int\limits_{\mathbb{R}^d} \int\limits_{S} \frac{ \alpha(\xi - \xi')\, Q(s, s') \, q(s')}{ v(s) \, q(s)} \,  d \xi' \, \nu(ds')  = 1.
$$
Consequently, the Markov jump process $X(t)$ with generator \eqref{LV} can be described in the following form:
starting with the state $X(0)=(\xi, s)$ the process jumps with the intensity $v(s)$, and the distribution of the new position of $X(t)$ has the density $ \frac{ \alpha(\xi - \xi')\, Q(s, s') \, q(s')}{ v(s) \, q(s)}$ with respect to the measure $l \otimes \nu$. Thus the coordinates $\xi'$ and $s'$ are conditionally independent under the initial condition $X(0)=(\xi, s)$, and we can write $X(t) = (\xi(t), \ s(t))$.

Denote by $\Theta(s,s')  =  \frac{ Q(s, s') \, q(s')}{ v(s) \, q(s)}$. Then the critical regime condition \eqref{rq-1} yields
\begin{equation}\label{tildeQ}
\int\limits_S \Theta(s,s') \,  \nu(ds')  = 1.
\end{equation}
The second component $s(t)$ of $X(t)$ is a continuous time Markov jump process in $S$ with generator
\begin{equation}\label{LS}
 {\mathcal{L}}_S \varphi(s) \ = \ v(s) \, \int\limits_{S} \Theta (s,s') \big( \varphi(s') - \varphi(s) \big)  \nu (ds').
\end{equation}

The first component $\xi(t)$ of $X(t)$ is not a continuous time Markov process (when $ v(s) \not \equiv 1$). Let us consider the sequence $\xi(0), \ \xi(1), \ldots$, where $\xi(n) \in \mathbb{R}^d$ is the first coordinate of $X(t)$ after the $n$-th jump. Then $\xi(n)$ is the random walk in $\mathbb{R}^d$, i.e. the sum of i.i.d. random variables with the common jump distribution equal to $\alpha(u)$. But the random time intervals between the jumps are not i.i.d.  random variables.


Applying the representation \eqref{nb} for $b(X(t), y)$ and taking $x=(\xi_0, s_0), \ y = (\xi_1, s_1)$ we have
\begin{equation}\label{Eb-1}
{\mathbb{E}}_x b(X(t), y) = {\mathbb{E}}_{(\xi_0, s_0)} \frac{\alpha(\xi(t)-\xi_1) \, Q(s(t), s_1)}{q(s(t))} \le \varkappa \, {\mathbb{E}}_{\xi_0} \alpha(\xi(t)-\xi_1),
\end{equation}
where $\varkappa = \frac{\max Q(s, s')}{q_{min}}$.
Using the decomposition for the distribution of the process $\xi(t)$ to the singular and regular parts we obtain
\begin{equation}\label{Eb-2}
 {\mathbb{E}}_{\xi_0} \alpha(\xi(t)-\xi_1)= e^{-v(s_0) t} \alpha(\xi_0 - \xi_1) + \int\limits_{\mathbb{R}^d}P(t, \xi, \xi_0) \alpha(\xi - \xi_1) \, d\xi,
\end{equation}
where $P(t, \xi, \xi_0)$ is the regular part of the distribution of $\xi(t)$:
 \begin{equation}\label{Eb-3}
P(t, \xi, \xi_0) = \sum\limits_{n=1}^{\infty} \alpha^{\ast n} (\xi - \xi_0) \, p_n(t).
\end{equation}
Here $\alpha^{\ast n}$ is the $n$-fold convolution of the function $\alpha$,  $p_n(t) = \Pr (n_X(t) = n)$ is the probability that the process $X(t)$ has $n$ jumps up to time $t$.
It is worth noting that
$$
\Pr (n_X(t) =n ) = \Pr (n_s (t) = n),
$$
where $n_s(t)$ is the number of jumps of the process $s(t)$ on the time interval $[0,t]$. Later, see \eqref{Pv}, we denote by $P_{v(\cdot)}$ the measure on the space of integer-valued trajectories corresponding to the process $n_s(t)$.

We estimate  $\alpha^{\ast n}$ and $p_n(t)$ separately.

\begin{lemma}\label{Eb-L1}
Assume that $\alpha(\xi) \ge 0, \;  \alpha(\xi)\in L^1({\mathbb{R}^d}) \cap L^{\infty}({\mathbb{R}^d})$ and $\int \alpha(\xi) d\xi =1$.
Then the following upper bound is valid:
\begin{equation}\label{alpha-n}
\alpha^{\ast n}(\xi) \ \le \ \frac{K}{n^{d/2}}.
\end{equation}
\end{lemma}

\begin{proof}
Considering $\alpha(\cdot)$ as the distribution of a random variable we conclude that the corresponding characteristic function $\varphi(k)$ satisfies the following properties:
$$
\varphi \in L^2({\mathbb{R}^d}) \cap C_0({\mathbb{R}^d}), \quad \varphi(0)=1, \; |\varphi(k)|<1, \ k \neq 0,
$$
where $ C_0({\mathbb{R}^d})$ is the space of continuous functions vanishing at infinity: $|\varphi(k)| \to 0$ as $|k| \to \infty$. From the statement of Lemma 1.5 \cite{Petrov}  it follow that there exist
$\delta>0$ and $\gamma>0$ such that
 \begin{equation}\label{alpha-n-1}
 |\varphi(k)| \le e^{-\gamma\, k^2} \quad \mbox{for all } \; |k| \le \delta.
 \end{equation}
Moreover, the properties of $\varphi$ imply that
\begin{equation}\label{alpha-n-2}
 |\varphi(k)| \le C \quad \mbox{with } \; 0<C<1 \quad \mbox{for all } \; |k| > \delta.
 \end{equation}
Since $\varphi^n (k)$ is the characteristic function of $\alpha^{\ast n}(\xi)$, then using the inverse Fourier transform together with \eqref{alpha-n-1} -  \eqref{alpha-n-2} we obtain the following uniform upper bound
$$
\alpha^{\ast n}(\xi) = \frac{1}{(2 \pi)^d} \int\limits_{\mathbb{R}^d} e^{- i k u} \varphi^n (k)\, d k
$$
$$
 \le \frac{1}{(2 \pi)^d} \Big( \int\limits_{|k| \le \delta} |\varphi (k)|^n\, d k + \int\limits_{|k| > \delta} |\varphi(k)|^n\, d k \Big)
$$
$$
\le \frac{1}{(2 \pi)^d} \int\limits_{\mathbb{R}^d} e^{- \gamma k^2 n}\, d k + \frac{1}{(2 \pi)^d} \, C^{n-2} \int\limits_{\mathbb{R}^d} |\varphi(k)|^2 \, dk
$$
$$
\le \ \frac{\hat C}{n^{d/2}} \ + \ \frac{C^{n-2}}{(2 \pi)^d}\, \|\varphi\|^2_{L^2} \ \le \ \frac{K}{n^{d/2}}.
$$
Here $\hat C, K$ are constants, and $0<C<1$ is the same constant as in \eqref{alpha-n-2}.
Thus estimate \eqref{alpha-n} is proved.
\end{proof}
\medskip

Denote by $F_{\alpha_j}(t), \ j=1,2,$ distribution functions of random variables with the exponential distribution with parameters $\alpha_1>0$ and $\alpha_2>0$ respectively, i.e. $F_{\alpha_j}(t) = 1 - e^{-\alpha_j t}, \ t \ge 0$. If $\alpha_2 > \alpha_1$, then
$$
F_{\alpha_2}(t)> F_{\alpha_1}(t) \;\; \forall \ t>0.
$$
We will use notation $P_{v(s(\tau))} (n(t)\le k)= P_{v(\cdot)} (n(t)\le k)$ for the Poisson process with intensity $v(s(\tau))$ depending on the Markov jump process $s(\tau)$ with generator  \eqref{LS}.  It is clear that
\begin{equation}\label{Pv}
\Pr (n_X(t)\le k)\ = \ P_{v(\cdot)} (n(t)\le k).
\end{equation}

\begin{proposition}\label{Prop5}
Assume that $0< \lambda_0 \le v(s) \le \lambda_1 < \infty$, then
\begin{equation}\label{1}
P_{v(\cdot)} (n(t)\le k) \ \le \ P_{\lambda_0} (n(t)\le k) \; \; \forall \ k \in \mathbb{N}.
\end{equation}
\end{proposition}

\begin{proof}
Inequality \eqref{1} is equivalent to the following inequality
\begin{equation}\label{2}
P_{v(\cdot)} (n(t)\ge k) \ \ge  \ P_{\lambda_0} (n(t)\ge k) \; \; \forall \ k \in \mathbb{N}.
\end{equation}
The assumption of the Proposition implies that $F_{v(\cdot)}(t) \ge F_{\lambda_0}(t) \; \forall \ t \ge 0$.
Recall that the convolution of distribution functions is defined as follows
$$
(F_1 \ast F_2) (t) = \int\limits_{- \infty}^{+ \infty} F_1(t-x) F_2(dx).
$$
Then, using that $F_1 \ast F_2 = F_2 \ast F_1$, we conclude that for any $s_i, \ s_j$
$$
(F_{v(s_i)} \ast F_{v(s_j)})(t) \ge (F_{\lambda_0} \ast F_{v(s_j)})(t) =  (F_{v(s_j)} \ast F_{\lambda_0})(t) \ge  (F_{\lambda_0} \ast  F_{\lambda_0})(t).
$$
Taking into account \eqref{tildeQ} we obtain for any $v_1 = v(s_1), \; s_1 = s(0)$ the following inequality
$$
P_{v(\cdot)} (n(t)\ge k) = \int\limits_{S^{k-1}} \Theta(s_1, s_2) \ldots \Theta(s_{k-1}, s_k)(F_{v_1}\ast F_{v_2} \ast \ldots \ast F_{v_k})(t) \ \nu(d s_2) \ldots \nu(d s_k)
$$
$$
 \ge \  F_{\lambda_0}^{\ast k}(t)  \ =  \ P_{\lambda_0} (n(t)\ge k),
$$
where $v_j=v(s_j) \ge \lambda_0 \; \; \forall \ j = 1, \ldots, k$. We used here that for a given states $s_1, \ldots, s_k$ of the process $s(t)$ the time intervals $t_1, \ldots, t_k$ of the waiting times of the corresponding jump are conditionally independent and have exponential distributions with parameters $v(s_1), \ldots, v(s_k)$ respectively.

Thus inequalities \eqref{2} and \eqref{1} are proved.
\end{proof}

Next we exploit estimates \eqref{alpha-n} and \eqref{1} to obtain the following upper bound for the regular part $P(t, \xi, \xi_0)$ of the distribution $\xi(t)$.

\begin{lemma}\label{Eb-L3}
The following upper bound holds for $P(t, \xi, \xi_0)$ and for all $t>0$
 \begin{equation}\label{Eb-l3}
P(t, \xi, \xi_0) \ \le \ \min \big\{ 1, \  \frac{\tilde K}{t^{d/2}} \big\}
\end{equation}
\end{lemma}

\begin{proof}
To get the upper bound on $P(t, \xi, \xi_0)$ we divide the sum in \eqref{Eb-3} into two parts:
\begin{equation}\label{AAA}
\sum\limits_{n=1}^{\infty} \alpha^{\ast n} (\xi) \, p_n(t) = \sum\limits_{n=1}^{[ \frac12 \lambda_0 t ]} \alpha^{\ast n} (\xi) \, p_n(t) + \sum\limits_{n = [ \frac12 \lambda_0 t]+1}^{\infty} \alpha^{\ast n} (\xi) \, p_n(t),
\end{equation}
where $\lambda_0>0$ is the same as in Proposition \ref{Prop5}, and estimate each of the sum in the right-hand side of \eqref{AAA} separately.
The boundedness of $\alpha(\cdot)$ together with the normalization condition \eqref{norm} imply that
\begin{equation}\label{A1}
M \ = \ \sup\limits_{k \in \mathbb{N}, \, \xi \in \mathbb{R}^d} \alpha^{\ast k}(\xi) \ = \ \|\alpha \|_{\infty}
\end{equation}
and
\begin{equation}\label{A2}
\sup\limits_{\xi} \alpha^{\ast (n+k)}(\xi) \ \le \ \sup\limits_{\xi} \alpha^{\ast n}(\xi) \quad \forall \ k=1,2, \ldots,
\end{equation}
Using bound \eqref{1} and the Stirling formula we get for the first sum in \eqref{AAA}:
$$
\sum\limits_{n=1}^{[ \frac12 \lambda_0 t ]} \alpha^{\ast n} (\xi) \, p_n(t) \le M \, \sum\limits_{n=1}^{[ \frac12 \lambda_0 t ]} p_n(t) \le M \, \Pr \Big( n_X(t) \le  [ \frac12 \lambda_0 t ] \Big)
$$
\begin{equation}\label{AAA-1}
\le M \, P_{\lambda_0} \Big( n(t) \le  [ \frac12 \lambda_0 t ] \Big) = M \, \sum\limits_{n=0}^{ [\frac12 \lambda_0 t]} \frac{( \lambda_0 t)^n}{n!} \, e^{- \lambda_0 t} \le \tilde M \, t \, e^{-B \, \lambda_0 \, t}
\end{equation}
with positive $B=\frac{1- \ln 2}{2}$, and $ \tilde M = \frac12 M \lambda_0$.

To estimate the second sum in \eqref{AAA} we exploit the bounds \eqref{alpha-n},  \eqref{A2} and the inequality
$\sum\limits_{n= [ \frac12 \lambda_0 t ]+1}^{\infty} p_n(t) <1$. Then we have
\begin{equation}\label{AAA-2}
\sum\limits_{n= [ \frac12  \lambda_0 t ]+1}^{\infty} \alpha^{\ast n} (\xi) \, p_n(t) \le
\sup\limits_{n > [ \frac12  \lambda_0 \, t ]} \sup\limits_{\xi} \, \alpha^{\ast n} (\xi) \le \frac{K}{\big( 1 +\big[ \frac{\lambda_0 \, t}{2}\big] \big)^{d/2}}.
\end{equation}
Finally from \eqref{AAA-1} and \eqref{AAA-2} we obtain the statement \eqref{Eb-l3} of Lemma \ref{Eb-L3}.
\end{proof}

Collecting \eqref{Eb-1}, \eqref{Eb-2} and \eqref{Eb-l3}  we get the bound \eqref{Eb-0}. Lemma \ref{main-tr} and Theorem \ref{T-multi} are completely proved.
\end{proof}

\end{proof}

\section{Appendix}

We can include in our model a possibility to jump. The analogous model in $\X$ has been considered earlier in \cite{KKP,KKPZ}.
More precisely, let us consider the following heuristic generator
$L\ + \ L_J$, where $L$  was defined by (\ref{generator}),
\begin{equation}\label{LJ}
L_J F(\gamma)\ = \  \int\limits_{\XXX} \sum_{x\in \gamma}
J(y,x) \Big( F ((\gamma\setminus x) \cup y)-F(\gamma) \Big) \, m(dy).
\end{equation}
Suppose that the total jump rate $\int J(y,x) m(dy)$
is uniformly bounded in $x$:
\begin{equation}\label{2aJ-bis}
\sup\limits_x \ \int\limits_{\XXX} J(y,x) m( dy) \ < \ C.
\end{equation}
Then the modified critical regime condition reads
\begin{equation}\label{criJ}
\int\limits_{\XXX} \big(a(x,y)+J(x,y)\big) \Psi(y) \ m(dy) = \big( V(x) + \int\limits_{\XXX} J(y,x) m(dy) \big) \Psi(x).
\end{equation}
Again it means that if the initial density of particles is equal to $\Psi(x)$, then this density is conserved.

We take $b(x,y)$ and $\bar{\mathfrak{m}}(dy)$ in the same way as it was defied by \eqref{nb}.
Then the generator ${\mathcal{L}}_J$ of the jump process takes the form
\begin{equation}\label{LambdaJ}
{\mathcal{L}}_J f(x) =  \int\limits_{\XXX} \Big( b(x,y) + \frac{J(x,y)}{\Psi(x)} \Big) ( f(y) - f(x)) \bar{\mathfrak{m}}(dy),
\end{equation}
and "transience" condition \eqref{2b} now can be written as
\begin{equation}\label{2bJ}
\sup\limits_{x,y} \ \int\limits_{0}^{\infty} \mathbb{E}_{x,y} b( \tilde X(t), \tilde Y(t)) dt < H,
\end{equation}
where $\tilde X(t)$ and $\tilde Y(t)$ are two independent copies of the Markov process with generator ${\mathcal{L}}_J$ given by \eqref{LambdaJ}.



\end{document}